\documentclass[twoside,onecolumn,centertags,final,a4paper,12pt,reqno]{amsart}
\usepackage{amsmath}
\usepackage{amsfonts}
\usepackage[left=1.5cm, bottom=3cm]{geometry}
\usepackage{cite}

\newtheorem{thm}{Theorem}
\newtheorem{lem}{Lemma}
\newtheorem{cor}{Corollary}
\theoremstyle{definition}
\newtheorem{dfn}{Definition}
\theoremstyle{remark}
\newtheorem{remark}{Remark}
\numberwithin{equation}{section}
\allowdisplaybreaks

\begin{document}
\title[Coefficient bounds for analytic functions]{Comprehensive subclasses
of analytic functions and coefficient bounds}
\author{Serap BULUT}
\address{Kocaeli University, Faculty of Aviation and Space Sciences,
Arslanbey Campus, 41285 Kartepe-Kocaeli, TURKEY}
\email{serap.bulut@kocaeli.edu.tr}

\begin{abstract}
In this paper, we introduce two general subclasses of analytic functions by
means of the principle of subordination and investigate the coefficient
bounds for functions in theese classes. The well-known results are obtained
as a corollary of our main results. Especially, we improve the results of
Alt\i nta\c{s} and K\i l\i \c{c} \cite{AK}.
\end{abstract}

\maketitle

\section{Definitions and Preliminaries}

Let $\mathcal{A}$ be the family of functions of the form%
\begin{equation}
f(z)=z+\sum_{n=2}^{\infty }a_{n}z^{n}  \label{1.1}
\end{equation}%
which are analytic in the open unit disk $\mathbb{D}=\left\{ z:z\in \mathbb{%
C\quad }\text{and}\mathbb{\quad }\left\vert z\right\vert <1\right\} .$

For analytic functions $f$ and $g$ with $f\left( 0\right) =g\left( 0\right) $%
, $f$ is said to be subordinate to $g$ in $\mathbb{D}$ if there exists an
analytic function $\mathfrak{h}$ on $\mathbb{D}$ such that%
\begin{equation*}
\mathfrak{h}\left( 0\right) =0,\quad \left\vert \mathfrak{h}\left( z\right)
\right\vert <1\text{\quad and\quad }f\left( z\right) =g\left( \mathfrak{h}%
\left( z\right) \right) \text{\qquad }\left( z\in \mathbb{D}\right) .
\end{equation*}%
We denote the subordination by%
\begin{equation*}
f\left( z\right) \prec g\left( z\right) \qquad \left( z\in \mathbb{D}\right)
.
\end{equation*}%
Note that if the function $g$ is univalent in $\mathbb{D}$, then we have%
\begin{equation*}
f\left( z\right) \prec g\left( z\right) \quad \left( z\in \mathbb{D}\right)
\Leftrightarrow f\left( 0\right) =g\left( 0\right) \text{\quad and\quad }%
f\left( \mathbb{D}\right) \subset g\left( \mathbb{D}\right) .
\end{equation*}

Let $\mathcal{N}$ be the class consisting of analytic and univalent
functions $\varphi :\mathbb{D}\rightarrow \mathbb{C}$ such that $\varphi (%
\mathbb{D})$ is convex with%
\begin{equation*}
\varphi (0)=1\qquad \text{and\qquad }\Re \left( \varphi \left( z\right)
\right) >0\quad \left( z\in \mathbb{D}\right) .
\end{equation*}%
By means of functions belong to the class $\mathcal{N}$ and the principle of
subordination, we consider following subclasses of analytic function class $%
\mathcal{A}$:%
\begin{eqnarray}
\mathcal{S}^{\ast }\left( \varphi \right)  &=&\left\{ f\in \mathcal{A}:\frac{%
zf^{\prime }\left( z\right) }{f\left( z\right) }\prec \varphi \left(
z\right) \quad \left( \varphi \in \mathcal{N};\;z\in \mathbb{D}\right)
\right\} ,  \label{1.a} \\
\mathcal{K}\left( \varphi \right)  &=&\left\{ f\in \mathcal{A}:1+\frac{%
zf^{\prime \prime }\left( z\right) }{f^{\prime }\left( z\right) }\prec
\varphi \left( z\right) \quad \left( \varphi \in \mathcal{N};\;z\in \mathbb{D%
}\right) \right\} ,  \label{1.b} \\
\mathcal{C}\left( \varphi ,\psi \right)  &=&\left\{ f\in \mathcal{A}:g\in
\mathcal{K}\left( \psi \right) \;\wedge \;\frac{f^{\prime }\left( z\right) }{%
g^{\prime }\left( z\right) }\prec \varphi \left( z\right) \quad \left(
\varphi ,\psi \in \mathcal{N};\;z\in \mathbb{D}\right) \right\} ,
\label{1.c} \\
\mathcal{CS}\left( \varphi ,\psi \right)  &=&\left\{ f\in \mathcal{A}:g\in
\mathcal{S}^{\ast }\left( \psi \right) \;\wedge \;\frac{f\left( z\right) }{%
g\left( z\right) }\prec \varphi \left( z\right) \quad \left( \varphi ,\psi
\in \mathcal{N};\;z\in \mathbb{D}\right) \right\} ,  \label{1.d} \\
\mathcal{QK}\left( \varphi ,\psi \right)  &=&\left\{ f\in \mathcal{A}:g\in
\mathcal{K}\left( \psi \right) \;\wedge \;\frac{\left( zf^{\prime }\left(
z\right) \right) ^{\prime }}{g^{\prime }\left( z\right) }\prec \varphi
\left( z\right) \quad \left( \varphi ,\psi \in \mathcal{N};\;z\in \mathbb{D}%
\right) \right\} .  \label{1.e}
\end{eqnarray}

The classes $\mathcal{S}^{\ast }\left( \varphi \right) $ and $\mathcal{K}%
\left( \varphi \right) $ are introduced by Ma and Minda \cite{Ma-Minda}, and
the class $\mathcal{C}\left( \varphi ,\psi \right) $ is introduced by Kim et
al. \cite{KCS}. Since%
\begin{equation*}
f(z)\in \mathcal{K}\left( \varphi \right) \Leftrightarrow zf^{\prime }(z)\in
\mathcal{S}^{\ast }\left( \varphi \right) ,
\end{equation*}%
we also have%
\begin{equation*}
f(z)\in \mathcal{C}\left( \varphi ,\psi \right) \Leftrightarrow \exists g\in
\mathcal{S}^{\ast }\left( \psi \right) \quad \text{s.t.}\quad \frac{%
zf^{\prime }\left( z\right) }{g\left( z\right) }\prec \varphi \left(
z\right) \quad \left( z\in \mathbb{D}\right) .
\end{equation*}

\begin{remark}
If we choose%
\begin{equation*}
\varphi \left( z\right) =\frac{1+Az}{1+Bz}\quad \left( -1\leq B<A\leq
1\right)
\end{equation*}%
in $\left( \ref{1.a}\right) $ and $\left( \ref{1.b}\right) $, then we get
the classes of Janowski starlike functions and Janowski convex functions%
\begin{equation*}
\mathcal{S}^{\ast }\left( \frac{1+Az}{1+Bz}\right) =\mathcal{S}^{\ast
}\left( A,B\right) \qquad \text{and\qquad }\mathcal{K}\left( \frac{1+Az}{1+Bz%
}\right) =\mathcal{K}\left( A,B\right) ,
\end{equation*}%
respectively, introduced by Janowski \cite{J}.
\end{remark}

\begin{remark}
If we choose%
\begin{equation*}
\varphi \left( z\right) =\frac{1+Az}{1+Bz}\quad \left( -1\leq B<A\leq
1\right) \qquad \text{and}\qquad \psi \left( z\right) =\frac{1+z}{1-z}
\end{equation*}%
in $\left( \ref{1.c}\right) $ and $\left( \ref{1.d}\right) $, then we obtain
the classes%
\begin{equation*}
\mathcal{C}\left( \frac{1+Az}{1+Bz},\frac{1+z}{1-z}\right) =\mathcal{CCV}%
\left( A,B\right) ,\qquad \mathcal{CS}\left( \frac{1+Az}{1+Bz},\frac{1+z}{1-z%
}\right) =\mathcal{CST}\left( A,B\right)
\end{equation*}%
introduced by Reade \cite{Reade}; and from $\left( \ref{1.e}\right) $, we
have the class%
\begin{equation*}
\mathcal{QK}\left( \frac{1+Az}{1+Bz},\frac{1+z}{1-z}\right) =\mathcal{QCV}%
\left( A,B\right)
\end{equation*}%
introduced by Alt\i nta\c{s} and K\i l\i \c{c} \cite{AK}.
\end{remark}

\begin{remark}
If we choose%
\begin{equation*}
\varphi \left( z\right) =\frac{1+\left( 1-2\alpha \right) z}{1-z}\quad
\left( 0\leq \alpha <1\right) \quad \text{and}\quad \psi \left( z\right) =%
\frac{1+\left( 1-2\beta \right) z}{1-z}\quad \left( 0\leq \beta <1\right)
\end{equation*}%
in $\left( \ref{1.c}\right) $, then we obtain the class of close-to-convex
functions of order $\alpha $ and type $\beta ,$%
\begin{equation*}
\mathcal{C}\left( \frac{1+\left( 1-2\alpha \right) z}{1-z},\frac{1+\left(
1-2\beta \right) z}{1-z}\right) =\mathcal{C}(\alpha ,\beta ),
\end{equation*}%
introduced by Libera \cite{L}.
\end{remark}

\begin{remark}
If we choose%
\begin{equation*}
\varphi \left( z\right) =\frac{1+z}{1-z}=\psi \left( z\right)
\end{equation*}%
in $\left( \ref{1.a}\right) $-$\left( \ref{1.c}\right) $, then we get the
familiar class $\mathcal{S}^{\ast }$ consists of starlike functions in $%
\mathbb{D}$, $\mathcal{K}$ consists of convex functions in $\mathbb{D}$ and $%
\mathcal{C}$ consists of close-to-convex function in $\mathbb{D}$,
respectively. Also, from $\left( \ref{1.d}\right) $ and $\left( \ref{1.e}%
\right) $, we get the class $\mathcal{CS}$ of close-to-starlike functions in
$\mathbb{D}$ introduced by Reade \cite{Reade}, and the class $\mathcal{Q}$
of quasi-convex functions in $\mathbb{D}$ introduced by Noor and Thomas \cite%
{NT}, respectively.
\end{remark}

Throughout this paper%
\begin{equation*}
0\leq \delta \leq \lambda \leq 1\qquad \text{and\qquad }\varphi ,\psi \in
\mathcal{N}.
\end{equation*}

Now we define new comprehensive subclasses of analytic function class $%
\mathcal{A}$, as follows:

\begin{dfn}
\label{def3}A function $f\in \mathcal{A}$ is said to be in the class $%
\mathcal{K}_{\lambda ,\delta }\left( \varphi ,\psi \right) $ if%
\begin{equation}
\frac{f^{\prime }\left( z\right) +\left( \lambda -\delta +2\lambda \delta
\right) zf^{\prime \prime }\left( z\right) +\lambda \delta z^{2}f^{\prime
\prime \prime }(z)}{g^{\prime }\left( z\right) }\prec \varphi \left(
z\right) \qquad \left( z\in \mathbb{D}\right) ,  \label{def.3}
\end{equation}%
where $g\in \mathcal{K}\left( \psi \right) .$
\end{dfn}

\begin{dfn}
\label{def4}a function $f\in \mathcal{A}$ is said to be in the class $%
\mathcal{S}_{\lambda ,\delta }\left( \varphi ,\psi \right) $ if%
\begin{equation}
\frac{\left( 1-\lambda +\delta \right) f\left( z\right) +\left( \lambda
-\delta \right) zf^{\prime }\left( z\right) +\lambda \delta z^{2}f^{\prime
\prime }(z)}{g\left( z\right) }\prec \varphi \left( z\right) \qquad \left(
z\in \mathbb{D}\right) ,  \label{def.4}
\end{equation}%
where $g\in \mathcal{S}^{\ast }\left( \psi \right) .$
\end{dfn}

\begin{remark}
If we set $\delta =0$ and $\lambda =1$ in Definition $\ref{def3}$ and
Definition $\ref{def4}$, then we have the classes%
\begin{equation*}
\mathcal{K}_{1,0}\left( \varphi ,\psi \right) =\mathcal{QK}\left( \varphi
,\psi \right) \qquad \text{and\qquad }\mathcal{S}_{1,0}\left( \varphi ,\psi
\right) =\mathcal{C}\left( \varphi ,\psi \right) .
\end{equation*}%
Also when $\delta =0$ and $\lambda =0$, we get the classes%
\begin{equation*}
\mathcal{K}_{0,0}\left( \varphi ,\psi \right) =\mathcal{C}\left( \varphi
,\psi \right) \qquad \text{and\qquad }\mathcal{S}_{0,0}\left( \varphi ,\psi
\right) =\mathcal{CS}\left( \varphi ,\psi \right) .
\end{equation*}
\end{remark}

\begin{remark}
If we set $\delta =0$ and%
\begin{equation*}
\varphi \left( z\right) =\frac{1+Az}{1+Bz}\quad \left( -1\leq B<A\leq
1\right) \qquad \text{and}\qquad \psi \left( z\right) =\frac{1+z}{1-z}
\end{equation*}%
in Definition $\ref{def3}$ and Definition $\ref{def4}$, then we obtain the
classes $\mathcal{Q}_{\mathcal{CV}}\left( \lambda ,A,B\right) $ and $%
\mathcal{Q}_{\mathcal{ST}}\left( \lambda ,A,B\right) $, respectively,
introduced very recently by Alt\i nta\c{s} and K\i l\i \c{c} \cite{AK}.
These classes consist of functions $f\in \mathcal{A}$ satisfying%
\begin{equation*}
\frac{f^{\prime }\left( z\right) +\lambda zf^{\prime \prime }\left( z\right)
}{g^{\prime }\left( z\right) }\prec \frac{1+Az}{1+Bz}\qquad \left( g\in
\mathcal{K},\;z\in \mathbb{D}\right)
\end{equation*}%
and%
\begin{equation*}
\frac{\left( 1-\lambda \right) f\left( z\right) +\lambda zf^{\prime }\left(
z\right) }{g\left( z\right) }\prec \frac{1+Az}{1+Bz}\qquad \left( g\in
\mathcal{S}^{\ast },\;z\in \mathbb{D}\right) ,
\end{equation*}%
respectively.
\end{remark}

Alt\i nta\c{s} and K\i l\i \c{c} \cite{AK} obtained following coefficient
bounds for functions belong to the classes $\mathcal{Q}_{\mathcal{CV}}\left(
\lambda ,A,B\right) $ and $\mathcal{Q}_{\mathcal{ST}}\left( \lambda
,A,B\right) $, as follows:

\begin{thm}
\label{thm.A}If $f\in \mathcal{Q}_{\mathcal{CV}}\left( \lambda ,A,B\right) ,$
then%
\begin{equation*}
\left\vert a_{n}\right\vert \leq \frac{1}{1+\left( n-1\right) \lambda }%
\left( 1+\frac{\left( n-1\right) \left( A-B\right) }{1-B}\right) \qquad
\left( n=2,3,\ldots \right) .
\end{equation*}
\end{thm}

\begin{thm}
\label{thm.B}If $f\in \mathcal{Q}_{\mathcal{ST}}\left( \lambda ,A,B\right) ,$
then%
\begin{equation*}
\left\vert a_{n}\right\vert \leq \frac{n}{1+\left( n-1\right) \lambda }%
\left( 1+\frac{\left( n-1\right) \left( A-B\right) }{1-B}\right) \qquad
\left( n=2,3,\ldots \right) .
\end{equation*}
\end{thm}

In this work, we obtain coefficient bounds for functions in the
comprehensive subclasses $\mathcal{K}_{\lambda ,\delta }\left( \varphi ,\psi
\right) $ and $\mathcal{S}_{\lambda ,\delta }\left( \varphi ,\psi \right) $\
of analytic functions. Our results improve the results of Alt\i nta\c{s} and
K\i l\i \c{c} \cite{AK} (Theorem $\ref{thm.A}$ and Theorem $\ref{thm.B}$).

\section{Main results}

\begin{lem}
\label{lem2}\cite{R2} Let the function $\Phi $ given by%
\begin{equation*}
\Phi \left( z\right) =\sum_{n=1}^{\infty }A_{n}z^{n}\qquad \left( z\in
\mathbb{D}\right)
\end{equation*}%
be convex in $\mathbb{D}.$ Also let the function $\Psi $ given by%
\begin{equation*}
\Psi (z)=\sum_{n=1}^{\infty }B_{n}z^{n}\qquad \left( z\in \mathbb{D}\right)
\end{equation*}%
be holomorphic in $\mathbb{D}.$ If%
\begin{equation*}
\Psi \left( z\right) \prec \Phi \left( z\right) \qquad \left( z\in \mathbb{D}%
\right) ,
\end{equation*}%
then%
\begin{equation*}
\left\vert B_{n}\right\vert \leq \left\vert A_{1}\right\vert \qquad \left(
n=1,2,\ldots \right) .
\end{equation*}
\end{lem}

\begin{lem}
\label{lem3}\cite{XGS} Let $f\in \mathcal{K}\left( \psi \right) $ and be of
the form $\left( \ref{1.1}\right) $, then%
\begin{equation*}
\left\vert a_{n}\right\vert \leq \frac{\prod\limits_{j=0}^{n-2}\left(
j+\left\vert \psi ^{\prime }(0)\right\vert \right) }{n!}\qquad \left(
n=2,3,\ldots \right) .
\end{equation*}
\end{lem}

\begin{lem}
\label{lem4}\cite{XGS} Let $f\in \mathcal{S}^{\ast }\left( \psi \right) $
and be of the form $\left( \ref{1.1}\right) $, then%
\begin{equation*}
\left\vert a_{n}\right\vert \leq \frac{\prod\limits_{j=0}^{n-2}\left(
j+\left\vert \psi ^{\prime }(0)\right\vert \right) }{\left( n-1\right) !}%
\qquad \left( n=2,3,\ldots \right) .
\end{equation*}
\end{lem}

\begin{thm}
\label{thm1}Let $f\in \mathcal{K}_{\lambda ,\delta }\left( \varphi ,\psi
\right) $ and be of the form $(\ref{1.1})$, then%
\begin{eqnarray}
&&\left[ 1+\left( n-1\right) \left( \lambda -\delta +2\lambda \delta \right)
+\left( n-1\right) \left( n-2\right) \lambda \delta \right] \left\vert
a_{n}\right\vert  \notag \\
&\leq &\frac{\prod\limits_{j=0}^{n-2}\left( j+\left\vert \psi ^{\prime
}(0)\right\vert \right) }{n!}+\frac{\left\vert \varphi ^{\prime
}(0)\right\vert }{n}\left( 1+\sum_{k=1}^{n-2}\frac{\prod%
\limits_{j=0}^{n-k-2}\left( j+\left\vert \psi ^{\prime }(0)\right\vert
\right) }{\left( n-k-1\right) !}\right) \quad \left( n=2,3,\ldots \right) .
\label{2.1}
\end{eqnarray}
\end{thm}

\begin{proof}
Let the function $f\in \mathcal{K}_{\lambda ,\delta }\left( \varphi ,\psi
\right) $ be defined by $\left( \ref{1.1}\right) $. Therefore, by Definition
$\ref{def3}$, there exists a function%
\begin{equation}
g(z)=z+\sum_{n=2}^{\infty }b_{n}z^{n}\in \mathcal{K}\left( \psi \right)
,\;\psi \in \mathcal{M}  \label{2.2}
\end{equation}%
so that%
\begin{equation}
\frac{f^{\prime }\left( z\right) +\left( \lambda -\delta +2\lambda \delta
\right) zf^{\prime \prime }\left( z\right) +\lambda \delta z^{2}f^{\prime
\prime \prime }(z)}{g^{\prime }\left( z\right) }\prec \varphi \left(
z\right) \qquad \left( z\in \mathbb{D}\right) .  \label{2.3}
\end{equation}%
Note that by $\left( \ref{2.2}\right) $ and Lemma $\ref{lem3}$, we have%
\begin{equation}
\left\vert b_{n}\right\vert \leq \frac{\prod\limits_{j=0}^{n-2}\left(
j+\left\vert \psi ^{\prime }(0)\right\vert \right) }{n!}\qquad \left(
n=2,3,\ldots \right) .  \label{2.a}
\end{equation}%
Let us define the function $p(z)$ by%
\begin{equation}
p(z)=\frac{f^{\prime }\left( z\right) +\left( \lambda -\delta +2\lambda
\delta \right) zf^{\prime \prime }\left( z\right) +\lambda \delta
z^{2}f^{\prime \prime \prime }(z)}{g^{\prime }\left( z\right) }\qquad (z\in
\mathbb{D}).  \label{2.4}
\end{equation}%
Then according to $\left( \ref{2.3}\right) $ and $\left( \ref{2.4}\right) $,
we get%
\begin{equation}
p(z)\prec \varphi (z)\qquad (z\in \mathbb{D}).  \label{2.5}
\end{equation}%
Hence, using Lemma $\ref{lem2}$, we obtain%
\begin{equation}
\left\vert \frac{p^{\left( m\right) }\left( 0\right) }{m!}\right\vert
=\left\vert c_{m}\right\vert \leq \left\vert \varphi ^{\prime
}(0)\right\vert \qquad \left( m=1,2,\ldots \right) ,  \label{2.6}
\end{equation}%
where%
\begin{equation}
p(z)=1+c_{1}z+c_{2}z^{2}+\cdots \qquad (z\in \mathbb{D}).  \label{2.6.a}
\end{equation}%
Also from $\left( \ref{2.4}\right) $, we find%
\begin{equation}
f^{\prime }\left( z\right) +\left( \lambda -\delta +2\lambda \delta \right)
zf^{\prime \prime }\left( z\right) +\lambda \delta z^{2}f^{\prime \prime
\prime }(z)=p(z)g^{\prime }\left( z\right) .  \label{2.11}
\end{equation}%
Since $a_{1}=b_{1}=1$, in view of $\left( \ref{2.11}\right) $, we obtain%
\begin{eqnarray}
&&n\left[ 1+\left( n-1\right) \left( \lambda -\delta +2\lambda \delta
\right) +\left( n-1\right) \left( n-2\right) \lambda \delta \right]
a_{n}-nb_{n}  \notag \\
&=&c_{n-1}+2c_{n-2}b_{2}+\cdots +\left( n-1\right) c_{1}b_{n-1}  \notag \\
&=&\sum_{k=1}^{n-1}\left( n-k\right) c_{k}b_{n-k}\quad \left( n=2,3,\ldots
\right) .  \label{2.12}
\end{eqnarray}%
Now we get the desired result given in $\left( \ref{2.1}\right) $ by using $%
\left( \ref{2.a}\right) ,\left( \ref{2.6}\right) $ and $\left( \ref{2.12}%
\right) .$
\end{proof}

\begin{thm}
\label{thm2}Let $f\in \mathcal{S}_{\lambda ,\delta }\left( \varphi ,\psi
\right) $ and be of the form $(\ref{1.1})$, then%
\begin{eqnarray}
&&\left[ 1+\left( n-1\right) \left( \lambda -\delta +2\lambda \delta \right)
+\left( n-1\right) \left( n-2\right) \lambda \delta \right] \left\vert
a_{n}\right\vert  \notag \\
&\leq &\frac{\prod\limits_{j=0}^{n-2}\left( j+\left\vert \psi ^{\prime
}(0)\right\vert \right) }{\left( n-1\right) !}+\left\vert \varphi ^{\prime
}(0)\right\vert \left( 1+\sum_{k=1}^{n-2}\frac{\prod\limits_{j=0}^{n-k-2}%
\left( j+\left\vert \psi ^{\prime }(0)\right\vert \right) }{\left(
n-k-1\right) !}\right) \quad \left( n=2,3,\ldots \right) .  \label{3.1}
\end{eqnarray}
\end{thm}

\begin{proof}
Let the function $f\in \mathcal{S}_{\lambda ,\delta }\left( \varphi ,\psi
\right) $ be defined by $\left( \ref{1.1}\right) $. Therefore, by Definition
$\ref{def4}$, there exists a function%
\begin{equation}
g(z)=z+\sum_{n=2}^{\infty }b_{n}z^{n}\in \mathcal{S}^{\ast }\left( \psi
\right) ,\;\psi \in \mathcal{M}  \label{3.a}
\end{equation}%
so that%
\begin{equation}
\frac{\left( 1-\lambda +\delta \right) f\left( z\right) +\left( \lambda
-\delta \right) zf^{\prime }\left( z\right) +\lambda \delta z^{2}f^{\prime
\prime }(z)}{g\left( z\right) }\prec \varphi \left( z\right) \qquad \left(
z\in \mathbb{D}\right) .  \label{3.2}
\end{equation}%
Note that by $\left( \ref{3.a}\right) $ and Lemma $\ref{lem4}$, we have%
\begin{equation}
\left\vert b_{n}\right\vert \leq \frac{\prod\limits_{j=0}^{n-2}\left(
j+\left\vert \psi ^{\prime }(0)\right\vert \right) }{\left( n-1\right) !}%
\qquad \left( n=2,3,\ldots \right) .  \label{3.3}
\end{equation}%
Let us define the function $q(z)$ by%
\begin{equation}
q(z)=\frac{\left( 1-\lambda +\delta \right) f\left( z\right) +\left( \lambda
-\delta \right) zf^{\prime }\left( z\right) +\lambda \delta z^{2}f^{\prime
\prime }(z)}{g\left( z\right) }\qquad (z\in \mathbb{D}).  \label{3.4}
\end{equation}%
Then according to $\left( \ref{3.2}\right) $ and $\left( \ref{3.4}\right) $,
we get%
\begin{equation}
q(z)\prec \varphi \left( z\right) \qquad \left( z\in \mathbb{D}\right) .
\label{3.5}
\end{equation}%
Hence, using Lemma $\ref{lem2}$, we obtain%
\begin{equation}
\left\vert \frac{q^{\left( m\right) }\left( 0\right) }{m!}\right\vert
=\left\vert d_{m}\right\vert \leq \left\vert \varphi ^{\prime
}(0)\right\vert \qquad \left( m=1,2,\ldots \right) ,  \label{3.6}
\end{equation}%
where%
\begin{equation}
q(z)=1+d_{1}z+d_{2}z^{2}+\cdots \qquad (z\in \mathbb{D}).  \label{3.7}
\end{equation}%
Also from $\left( \ref{3.4}\right) $, we find%
\begin{equation}
\left( 1-\lambda +\delta \right) f\left( z\right) +\left( \lambda -\delta
\right) zf^{\prime }\left( z\right) +\lambda \delta z^{2}f^{\prime \prime
}(z)=q(z)g\left( z\right) .  \label{3.8}
\end{equation}%
Since $a_{1}=b_{1}=1$, in view of $\left( \ref{3.8}\right) $, we obtain%
\begin{eqnarray}
&&\left[ 1-\lambda +\delta +n\left( \lambda -\delta \right) +n\left(
n-1\right) \lambda \delta \right] a_{n}-b_{n}  \notag \\
&=&c_{n-1}+c_{n-2}b_{2}+\cdots +c_{1}b_{n-1}  \notag \\
&=&\sum_{k=1}^{n-1}c_{k}b_{n-k}\qquad \left( n=2,3,\ldots \right) .
\label{3.9}
\end{eqnarray}%
Now we get the desired result given in $\left( \ref{3.1}\right) $ by using $%
\left( \ref{3.3}\right) ,\left( \ref{3.6}\right) $ and $\left( \ref{3.9}%
\right) .$
\end{proof}

\section{Corollaries and Consequences}

Letting $\delta =0$ and $\lambda =1$ in Theorem $\ref{thm1}$ and Theorem $%
\ref{thm2}$, we obtain the following consequences, respectively.

\begin{cor}
\label{cor.x}Let $f\in \mathcal{QK}\left( \varphi ,\psi \right) $ and be of
the form $(\ref{1.1})$, then%
\begin{equation*}
\left\vert a_{n}\right\vert \leq \frac{\prod\limits_{j=0}^{n-2}\left(
j+\left\vert \psi ^{\prime }(0)\right\vert \right) }{n^{2}\left( n-1\right) !%
}+\frac{\left\vert \varphi ^{\prime }(0)\right\vert }{n^{2}}\left(
1+\sum_{k=1}^{n-2}\frac{\prod\limits_{j=0}^{n-k-2}\left( j+\left\vert \psi
^{\prime }(0)\right\vert \right) }{\left( n-k-1\right) !}\right) \qquad
\left( n=2,3,\ldots \right) .
\end{equation*}
\end{cor}

\begin{cor}
\label{cor.a}Let $f\in \mathcal{C}\left( \varphi ,\psi \right) $ and be of
the form $(\ref{1.1})$, then%
\begin{equation*}
\left\vert a_{n}\right\vert \leq \frac{\prod\limits_{j=0}^{n-2}\left(
j+\left\vert \psi ^{\prime }(0)\right\vert \right) }{n!}+\frac{\left\vert
\varphi ^{\prime }(0)\right\vert }{n}\left( 1+\sum_{k=1}^{n-2}\frac{%
\prod\limits_{j=0}^{n-k-2}\left( j+\left\vert \psi ^{\prime }(0)\right\vert
\right) }{\left( n-k-1\right) !}\right) \qquad \left( n=2,3,\ldots \right) .
\end{equation*}
\end{cor}

Letting $\delta =0$ and $\lambda =0$ in Theorem $\ref{thm2}$, we obtain the
following consequence.

\begin{cor}
Let $f\in \mathcal{CS}\left( \varphi ,\psi \right) $ and be of the form $(%
\ref{1.1})$, then%
\begin{equation*}
\left\vert a_{n}\right\vert \leq \frac{\prod\limits_{j=0}^{n-2}\left(
j+\left\vert \psi ^{\prime }(0)\right\vert \right) }{\left( n-1\right) !}%
+\left\vert \varphi ^{\prime }(0)\right\vert \left( 1+\sum_{k=1}^{n-2}\frac{%
\prod\limits_{j=0}^{n-k-2}\left( j+\left\vert \psi ^{\prime }(0)\right\vert
\right) }{\left( n-k-1\right) !}\right) \qquad \left( n=2,3,\ldots \right) .
\end{equation*}
\end{cor}

If we choose%
\begin{equation*}
\varphi \left( z\right) =\frac{1+\left( 1-2\alpha \right) z}{1-z}\quad
\left( 0\leq \alpha <1\right) \quad \text{and}\quad \psi \left( z\right) =%
\frac{1+\left( 1-2\beta \right) z}{1-z}\quad \left( 0\leq \beta <1\right)
\end{equation*}%
in Corollary $\ref{cor.a}$, then we get following consequence.

\begin{cor}
\cite{L} Let $f\in \mathcal{C}(\alpha ,\beta )\,\left( 0\leq \alpha ,\beta
<1\right) $ and be of the form $(\ref{1.1})$, then%
\begin{equation*}
\left\vert a_{n}\right\vert \leq \frac{2\left( 3-2\beta \right) \left(
4-2\beta \right) \cdots \left( n-2\beta \right) }{n!}\left[ n\left( 1-\alpha
\right) +\left( \alpha -\beta \right) \right] \qquad \left( n=2,3,\ldots
\right) .
\end{equation*}
\end{cor}

Letting%
\begin{equation*}
\delta =0,\quad \varphi \left( z\right) =\frac{1+Az}{1+Bz}\quad \left(
-1\leq B<A\leq 1\right) ,\quad \psi \left( z\right) =\frac{1+z}{1-z}
\end{equation*}%
in Theorem $\ref{thm1}$ and Theorem $\ref{thm2}$, we obtain the following
consequences, respectively.

\begin{cor}
\label{cor1}Let $f\in \mathcal{Q}_{\mathcal{CV}}\left( \lambda ,A,B\right) $
and be of the form $(\ref{1.1})$, then%
\begin{equation*}
\left\vert a_{n}\right\vert \leq \frac{1}{1+\left( n-1\right) \lambda }%
\left( 1+\frac{\left( n-1\right) \left( A-B\right) }{2}\right) \qquad \left(
n=2,3,\ldots \right) .
\end{equation*}
\end{cor}

\begin{cor}
\label{cor2}Let $f\in \mathcal{Q}_{\mathcal{ST}}\left( \lambda ,A,B\right) $
and be of the form $(\ref{1.1})$, then%
\begin{equation*}
\left\vert a_{n}\right\vert \leq \frac{n}{1+\left( n-1\right) \lambda }%
\left( 1+\frac{\left( n-1\right) \left( A-B\right) }{2}\right) \qquad \left(
n=2,3,\ldots \right) .
\end{equation*}
\end{cor}

\begin{remark}
It is clear that%
\begin{equation*}
1+\frac{\left( n-1\right) \left( A-B\right) }{2}\leq 1+\frac{\left(
n-1\right) \left( A-B\right) }{1-B}\quad \left( -1\leq B<A\leq
1,\;n=2,3,\ldots \right) ,
\end{equation*}%
which would obviously yield significant improvements of Theorem $\ref{thm.A}$
and Theorem $\ref{thm.B}$.
\end{remark}

Letting%
\begin{equation*}
\lambda =0,\quad A=1,\quad B=-1
\end{equation*}%
in Corollary $\ref{cor1}$ and Corollary $\ref{cor2}$, we have following
consequences, respectively.

\begin{cor}
\cite{Reade} Let $f\in \mathcal{C}$ and be of the form $(\ref{1.1})$, then%
\begin{equation*}
\left\vert a_{n}\right\vert \leq n\qquad \left( n=2,3,\ldots \right) .
\end{equation*}
\end{cor}

\begin{cor}
\cite{Reade} Let $f\in \mathcal{CS}$ and be of the form $(\ref{1.1})$, then%
\begin{equation*}
\left\vert a_{n}\right\vert \leq n^{2}\qquad \left( n=2,3,\ldots \right) .
\end{equation*}
\end{cor}

Letting%
\begin{equation*}
\lambda =1,\quad A=1,\quad B=-1
\end{equation*}%
in Corollary $\ref{cor1}$, we have following consequence.

\begin{cor}
\cite{N} Let $f\in \mathcal{Q}$ and be of the form $(\ref{1.1})$, then%
\begin{equation*}
\left\vert a_{n}\right\vert \leq 1\qquad \left( n=2,3,\ldots \right) .
\end{equation*}
\end{cor}


\begin{thebibliography}{99}
\bibitem{AK} Alt\i nta\c{s} O, K\i l\i \c{c} \"{O}\"{O}. Coefficient
estimates for a class containing quasi-convex functions. Turkish Journal of
Mathematics 2018; 42 (5): 2819--2825. doi:10.3906/mat-1805-90

\bibitem{J} Janowski W. Some extremal problems for certain families of
analytic functions I. Annales Polonici Mathematici 1973; 28: 297--326.

\bibitem{KCS} Kim YC, Choi JH, Sugawa T. Coefficient bounds and convolution
properties for certain classes of close-to-convex functions. Proceedings of
the Japan Academy (Series A) 2000; 76: 95--98.

\bibitem{L} Libera RJ. Some radius of convexity problems. Duke Mathematical
Jornal 1964; 31: 143--158.

\bibitem{Ma-Minda} Ma W, Minda D. A unified treatment ofsome special classes
of univalent functions. In: Proceedings of the Conference on Complex
Analysis; Cambridge, MA; 1992. pp.157--169.

\bibitem{N} Noor KI. On quasi-convex functions and related topics.
International Journal of Mathematics and Mathematical Sciences 1987; 10 (2):
241--258.

\bibitem{NT} Noor KI, Thomas DK. Quasi-convex univalent functions.
International Journal of Mathematics and Mathematical Sciences 1980; 3 (2):
255--266.

\bibitem{Reade} Reade MO. On close-to-convex univalent functions. Michigan
Mathematical Journal 1955; 3 (1): 59--62.

\bibitem{R2} Rogosinski W. On the coefficients of subordinate functions.
Proceedings of the London Mathematical Society (2) 1943; 48: 48--82.

\bibitem{XGS} Xu Q-H, Gui Y-C, Srivastava HM. Coefficient estimates for
certain subclasses of analytic functions of complex order. Taiwanese Journal
of Mathematics 2011; 15 (5): 2377--2386.
\end{thebibliography}
\end{document}